\documentclass[11pt, a4paper, oneside, reqno]{amsart}
\makeatletter
\def\@settitle{\begin{center}%
		\baselineskip14\p@\relax
		\normalfont\LARGE\scshape\bfseries
		\@title
	\end{center}%
}
\makeatother

\makeatletter

\def\subsection{\@startsection{subsection}{2}%
	\z@{.5\linespacing\@plus.7\linespacing}{.5\linespacing}%
	{\normalfont\large\bfseries}}

\def\subsubsection{\@startsection{subsubsection}{3}%
	\z@{.5\linespacing\@plus.7\linespacing}{.5\linespacing}%
	{\normalfont\itshape}}

\usepackage[colorlinks	= true,
			raiselinks	= true,
			linkcolor	= MidnightBlue,
			citecolor	= Mahogany,
			urlcolor	= ForestGreen,
			pdfauthor	= {Peyman Mohajerin Esfahani},
			pdftitle	= {},
			pdfkeywords	= {},   
			pdfsubject	= {},
			plainpages	= false]{hyperref}
\usepackage[usenames, dvipsnames]{color}
\usepackage{dsfont,amssymb,amsmath,subfig, graphicx,fancyhdr,mdframed}
\usepackage{amsfonts,dsfont,mathtools, mathrsfs,amsthm,wrapfig} 
\usepackage[]{siunitx}
\usepackage{ragged2e}
\usepackage{enumitem}

\allowdisplaybreaks
\usepackage{algorithm}
\usepackage[noend]{algpseudocode}

\allowdisplaybreaks
\date{\today}

\patchcmd\@setauthors
  {\MakeUppercase{\authors}}
  {\authors}
  {}{}
  
\newtheorem{Thm}{Theorem}[section]
\newtheorem{Prop}[Thm]{Proposition}

\newtheorem{Lem}[Thm]{Lemma}
\newtheorem{Cor}[Thm]{Corollary}
\newtheorem{As}[Thm]{Assumption}

\newtheorem{Prob}[Thm]{Problem}
 
\newtheorem{Rem}[Thm]{Remark}
\theoremstyle{remark}
\newtheorem{Ex}{Example}

\newcommand{\R}{\mathbb{R}}
\newcommand{\Rp}{\R_{>0}}

\newcommand{\N}{\mathbb{N}}

\newcommand{\Let}{\coloneqq}

\newcommand{\xdes}{x^{\rm d}}
\newcommand{\ydes}{y^{\rm d}}
\newcommand{\wdes}{w^{\rm d}}
\newcommand{\udes}{u^{\rm d}}
\newcommand{\eps}{\varepsilon}
\newcommand{\epsd}{\eps_{\mathrm{d}}}
\newcommand{\epsc}{\eps_{\mathrm{c}}}
\newcommand{\Mbdd}{A_b}
\newcommand{\Nbdd}{B_b}
\newcommand{\Sym}[1]{\left[{#1}\right]^	\dagger}
\newcommand{\BlkDiag}[1]{\mathfrak{Diag}\left\{{#1}\right\}}

\newcommand{\hb}{\bar{h}}
\newcommand{\QV}{\mathcal{Q}}
\newcommand{\Qtau}{\widetilde{\mathcal{Q}}}
\newcommand{\qt}{q}
\newcommand{\rhob}{k_b}
\newcommand{\Ac}{A_c}
\newcommand{\Bc}{B_c}
\newcommand{\Cc}{C_c}
\newcommand{\Dc}{D_c}
\newcommand{\xc}{w}
\newcommand{\xcl}{z}
\newcommand{\nx}{{n_x}}
\newcommand{\enu}{{n_u}}
\newcommand{\ny}{{n_y}}
\newcommand{\xstar}{x^{\star}}

\def\be{\begin{equation}}
\def\ee{\end{equation}}

\title[Event-triggered Robust Output Regulation]{Robust Output Regulation: Optimization-Based Synthesis and Event-Triggered Implementation}

\author[M. S. Sarafraz]{Mohammad Saeed Sarafraz}
\author[A. V. Proskurnikov]{Anton V. Proskurnikov}
\author[M. S. Tavazoei]{Mohammad Saleh Tavazoei}
\author[P. {Mohajerin Esfahani}]{Peyman {Mohajerin Esfahani}}

\thanks{The authors are with the Electrical Engineering Department, Sharif University of Technology, Iran ({\tt \{Sarafraz,Tavazoei\}@ee.sharif.ir}), Department of Electronics and Telecommunications, Politecnico di Torino, Italy and Institute for Problems in Mechanical Engineering of the Russian Academy of Sciences, St. Petersburg, Russia ({\tt anton.p.1982@ieee.org}), and the Delft Center for Systems and Control, TU Delft, The Netherlands ({\tt P.MohajerinEsfahani@tudelft.nl}).}

\begin{document}
\maketitle

\begin{abstract}
    We investigate the problem of practical output regulation, {i.e., to} design a controller that brings the system output in the vicinity of a desired target value while keeping the other variables bounded. We consider uncertain systems that are possibly nonlinear and the uncertainty of {their} linear part{s} is modeled element-wise through a parametric family of matrix boxes. {An optimization-based design procedure is proposed that delivers a continuous-time control and estimates the maximal regulation error. We also analyze an event-triggered emulation of this controller, which can be implemented on a digital platform, along with an explicit estimates of the regulation error.}
\end{abstract}
	
%===============================================================================
\section{Introduction} \label{sec:introduction}
%===============================================================================
	{
{O}{utput} regulation of uncertain dynamic systems is a fundamental problem in the control literature that finds a wide range of real-world applications~\cite{naldi2016robust}. The problem has been studied in various settings depending on the system dynamics (e.g., linear~\cite{kanev2004robust} or nonlinear \textcolor{black}{models}~\cite{liu2017event1}) and uncertainty nature (e.g., characterization in time~\cite{liu2017event} or frequency domains~\cite{doyle1982analysis}). In the light of recent developments of digitalization, communication and computation limitations of the controllers\textcolor{black}{'} architecture have also become an important consideration, which also contributes to this variety of the setting. In particular, one of the distinct features of the controllers is the time scale under which the controller receives output measurements or updates the control \textcolor{black}{efforts applied} to the systems~(e.g., continuous~\cite{kanev2004robust}, periodic~\cite{franklin1998digital}, or event-based interactions~\cite{xing2019event}). 

From a literature point of view, the uncertainty aspect is often the focus of robust control while the time-scale implementation of the controllers is the main theme of the event-triggered mechanism. The control synthesis tools of output regulation were first developed in the robust control literature for the setting in which the uncertainty is characterized in the frequency domain~\cite{doyle1982analysis,lanzon2008stability}. The setting of time-domain uncertainty, however, remains much less explored partly, due to the inherent provable computational difficulty~\cite{nemirovskii1993several}. \textcolor{black}{Considering} the current existing work\textcolor{black}{s} briefly mentioned above, we set the following as our main objective in this study:
	\begin{flushleft}    
	\justifying{\em {Given a nonlinear plant with element-wise time-domain uncertainty, we aim} to develop a scalable computational framework, along with rigorous and explicit performance guarantees, to synthesize a robust output {regulator and an} event-triggered mechanism {enabling its implementation on a digital platform}. 
	}
	\end{flushleft} 

\paragraph*{\bf Related literature on robust control}
A natural way \textcolor{black}{for modeling of the} uncertainty in the time domain is through the state-space representation of the dynamic \textcolor{black}{systems}. The stability of such systems can be cast as an optimization program, which unfortunately is often computationally intractable~\cite{gahinet1996affine}. Conservative approximations in the form of linear matrix inequalities (LMIs) are proposed for particular subclasses of uncertainty including single ellipsoid~\cite{tarbouriech2018insights} or polytopic systems with low number of vertices~\cite{kanev2004robust,karimi2007robust,agulhari2011lmi}. A richer modeling framework is element-wise or box uncertainty that allows to conveniently incorporate different sources of uncertainties. One approach to deal with this class of uncertainty is randomized algorithms~\cite{tempo2012randomized}. Alternatively, one can leverage the recent developments in the robust optimization literature~\cite{ben2003extended} to address the computational bottleneck. The optimization-based framework proposed in this paper exploits {the latter} result in the context of output regulation.

\paragraph*{\bf Related literature on {event-triggered control of} uncertain systems} The second part of this study is concerned with event-triggered control, \textcolor{black}{as} a powerful technique to address \textcolor{black}{the} potential communication limitation on the measurement or actuation side. A recent approach towards event-triggered control of uncertain systems builds on an adaptive control perspective~\cite{xing2019event,wang2019event}. 
The structure of an event triggering mechanism dictated by the necessity to maintain a positive dwell time between consecutive events usually makes it impossible to ensure asymptotic convergence. As such, the \emph{practical stability} (i.e., convergence to a ``tunable'' invariant set) is aimed for. Such a notion is also adopted in other contexts like quantized control~\cite{demirel2017optimal}, and has been investigated in the presence of a common Lyapunov function~\cite{liu2017event,liu2017event1}. 
		
Focusing on uncertain linear systems, the {work}~\cite{tarbouriech2018insights} considers norm-bounded uncertainties with continuous measurements, while \cite{liu2017robust,qian2019event} develop mechanism under the assumption that the system is minimum-phase. Most recently, the work~\cite{liang2020robust} studies the problem of output regulation together along with an event-triggering mechanism in which the robustness \textcolor{black}{is} guaranteed for an unstructured open uncertainty set. Concerning nonlinear systems, the recent work~\cite{khan2019new} proposes an event-triggered mechanism under the assumption that the system is input-to-state stable. Unlike the existing literature mention\textcolor{black}{ed} above, in this article we opt to introduce an event-triggering mechanism in which both monitoring the output measurement \textcolor{black}{and} implementing the actuation values operate on a discrete-time basis. To our best knowledge, none of the existing works considers this setting \textcolor{black}{in control of} uncertain nonlinear systems. The closest work in this spirit is~\cite{yang2018periodic}, \textcolor{black}{in which the class of} single-input single-output system \textcolor{black}{is considered} and the performance {is guaranteed only for sufficiently large feedback gains and sufficiently small periodic sampled-times}. 

\paragraph*{\bf Our contributions}
The particular emphasis of this study is on the computational aspect of the control design and the corresponding event-triggering mechanism, along with explicit performance guarantees. More specifically, the contributions of the article are summarized as follows:
 
\begin{enumerate}[label=$(\roman*)$]%, itemsep = 2mm, topsep = 1mm, leftmargin = 8mm]
	\item{\bf Dynamic structure and inherent hardness:} We propose a class of dynamic output controllers aiming to locate the closed-loop equilibrium in accordance with the desired regulation task (Section~\ref{subsec:equilibrium} and Lemma~\ref{lem.equilibrium}). We further show that from a computational viewpoint {stability analysis of the proposed controller} is strongly NP-hard~(Proposition~\ref{prop.interactability}).
	
    \item {\bf Robust control under element-wise (box) uncertainties:} We provide a sufficient condition along with an optimization framework to synthesize a dynamic output controller that enjoys a provable practical stability~(Theorem \ref{thm.Continuous}). As a byproduct, we also show that given any fixed controller, the proposed optimization program reduces to a tractable convex optimization that can be viewed as a computational certification tool for the practical stability~(Corollary~\ref{cor.nominaldesign}).

    \item {\bf Sampled-time event-triggered mechanism:}
	We propose a unifying triggering mechanism together with easy-to-compute sufficient conditions under which the proposed output controller can be implemented through aperiodic measurements and event-based actuation~(Theorem~\ref{thm.aperiodic}). The proposed mechanism offers explicit computable maximal inter-sampling and regulation error bounds. The proposed result subsumes both the existing approaches~\cite{tabuada2007event,heemels2013periodic} as a special case~({Corollary~\ref{cor.special} and }Remark~\ref{rem:comparison}).
	
	{\item {\bf Numerical algorithm:} Leveraging recent results from~\cite{lee2018sequential}, we propose a numerical algorithm to deal with nonlinearities of the proposed optimization program concerning the control synthesis of the output regulation task (Algorithm~\ref{alg:sequ}).}
\end{enumerate}

In the rest of the article, we present a formal description of the problem along with some basic assumptions in Section~\ref{sec:Problem Statement}. The robust control method is developed in Section~\ref{sec:Continuous Robust Control}, and the sampled-time event-triggered mechanism is presented in Section~\ref{sec:Triggered}. Section~\ref{Sec:Numerical} {discusses an algorithm to tackle the proposed optimization program, and further provides} {numerical example in order to validate the theoretical results.} 	
	
\paragraph*{\bf Notation}
The set of $n \times n$ symmetric matrices and the set of $n \times n$ positive-definite (semi-definite) symmetric ones are denoted by $\mathbb{S}^{n}$ and $\mathbb{S}^n_{\succ 0}$ ($\mathbb{S}^n_{\succeq 0}$), respectively. For two symmetric matrices $A$ \textcolor{black}{and} $B$, we write $A\succ B$ (respectively, $A \succeq B$) if $A-B\in\mathbb{S}^n_{\succ 0}$ (respectively, $\mathbb{S}^n_{\succeq 0}$). For a square matrix $A$, we denote $\Sym{A}=A+A^{\top}$. The symbol $\BlkDiag{A_1,A_2,...,A_n}$ denotes the block diagonal matrix with blocks $A_1,A_2,...,A_n$. For briefness in notations, the matrix $\begin{bsmallmatrix} A & B^{\top} \\ B & C \end{bsmallmatrix}$ is shown by $\begin{bsmallmatrix} A & * \\ B & C \end{bsmallmatrix}$.  We use $e_1,\ldots,e_m$ to denote the standard coordinate basis of $\R^m$. Also, $\mathds{1}_m\in\R^m$ denotes the vector whose elements are all equal to $1$.
%===============================================================================	
\section{Problem Statement}\label{sec:Problem Statement}
%===============================================================================
   Consider the control system
	\be\label{eq.sys-main}
	\left\{\begin{array}{l}
		\dot x(t)=A^{\star}x(t) + B^{\star}u(t) + k^\star\big(x(t)\big) \vspace{1mm}\\
		y(t)=Cx(t)
	\end{array}\right.
	\ee
	where {the vector} $x(t)\in\R^{\nx}$, $u(t)\in\R^{\enu}$, and $y(t)\in\R^{\ny}$ are the state, the control, and the output \textcolor{black}{vectors}, respectively. The matrices $A^{\star}$ and $B^{\star}$ represent the linear part of the state dynamics, and the function $k^\star:\R^\nx\rightarrow \R^\nx$ encapsulates the nonlinearity of the dynamics. Throughout this article, we assume that system~\eqref{eq.sys-main} admits a unique solution~$x(\cdot)$ {for any $x(0)$}. 
	{The controller to be designed in the next section has access only to the output $y(t)$.} {We allow} the matrices $A^{\star}, B^{\star}$ and the nonlinearity~$k^\star:\R^\nx \rightarrow \R^\nx$ in the system~\eqref{eq.sys-main} {to be partially uncertain}. Our main control objective is to stabilize~\eqref{eq.sys-main} in the Lagrange sense (i.e., all solutions are bounded) and steer the output trajectory of~\eqref{eq.sys-main} to an $\eps$-neighborhood of a target value $\ydes\in\R^\ny$. Formally speaking, we aim to ensure that 
	\be\label{eq.terminal}
	\sup_{t \ge 0} \| x(t) \| < \infty,\;\; 
	\varlimsup_{t\to\infty}\|y(t)-\ydes\| \le \eps\quad \forall x(0)\in\R^\nx.
	\ee
	The special case of $\eps = 0$ {corresponds to asymptotic output regulation} and {the relaxed condition with} is {known} as ``$\eps$-practical output stability"~\cite{moreau2000practical}.
	
	{Henceforth, the following assumptions are adopted.}
	
	\begin{As}
		{[Uncertainty characterization]}
		\label{as.coeff}
		{System~\eqref{eq.sys-main} and {the desired value} $\ydes \in \R^\ny$ satisfy the following \textcolor{black}{assumptions}}:  
		\begin{enumerate}[label=(\roman*)]
			\item \label{as.coeff.linear}
			(Box uncertainty) Matrices $A^{\star}$ and $B^{\star}$ {obey inequalities}
			\be\label{eq.assumption1}| A^\star - A | \le \Mbdd, \quad | B^\star - B | \le \Nbdd,
			\ee
			where {$A$ \textcolor{black}{and} $B$ are known nominal matrices}, the inequalities are understood element-wise, and $\Mbdd=\begin{bmatrix}a_{b_{ij}}\end{bmatrix}_{ij}, \Nbdd=\begin{bmatrix}b_{b_{ij}}\end{bmatrix}_{ij}$ are the respective uncertainty bounds.
			\item \label{as.coeff.nonlinear}
			{(Bounded nonlinearity)} The function~$k^\star$ satisfies
			\be \label{eq.assumption2}\|k^\star(x_1)-k^\star(x_2)\| \le \rhob, \qquad \forall x_1, x_2 \in \R^\nx
			\ee
			where $\rhob\geq0$ {is a known constant}. 
			\item \label{as.coeff.target}
			(Existence of an equilibrium) There exists a pair $(\xdes, \udes) \in \R^\nx\times\R^\enu$ such that
			\be\label{eq.eq-des}
			\ydes = C \xdes \quad \textrm{and} \quad A^\star \xdes + k^{\star}(\xdes) = - B^\star \udes\,.
			\ee
		\end{enumerate}
	\end{As}
	
	\textcolor{black}{ Assumption~\ref{as.coeff}\ref{as.coeff.nonlinear} holds if and only if the nonlinearity of the dynamics is globally bounded. If $\|k^{\star}(x)\|\leq C$, then~\eqref{eq.assumption2} holds with $\rhob=2C$. However, this estimate of $\rhob$ may be too conservative, e.g., if $k^{\star}$ is an uncertain constant, one can actually choose $\rhob=0$. The ``incremental'' condition~\eqref{eq.assumption2} thus provides more flexbility. There are several classes of nonlinear dynamics for which the bound~\eqref{eq.assumption2} is available: (i) \emph{pendulum-like} nonlinearity which represents periodicity of the dynamics, e.g., phase-locked loops~\cite{GLY,Smirnova:19}, or swing equations in power systems~\cite{esfahani2015tractable}; (ii)  nonlinearity presented due to an underlying \emph{neural network} architecture~\cite{fazlyab2020safety} or a lookup-table~\cite{Gao2010}. Such nonlinearities may or may not be fully known, but regardless of this knowledge, it is often too complicated to be utilized in control synthesis algorithms. Furthermore, we emphasize that the bound~$\rhob$ will not be required for control design and is only used in the final performance bounds.} 
	
	Assumption~\ref{as.coeff}\ref{as.coeff.target} involves $(\ny+\nx)$ algebraic constraints with $(\nx + \enu)$ variables. Therefore, we typically expect that such equations have a solution~$(\xdes,\udes)$ when $\enu \ge \ny$, i.e., the number of control variables is not less than the number of output\textcolor{black}{s}. When the dynamic \textcolor{black}{system}~\eqref{eq.sys-main} is linear (i.e., $k^*$ is constant), these equations reduce to a set of linear constraints, and that a sufficient condition for Assumption~\ref{as.coeff}\ref{as.coeff.target} is the matrix~$\begin{bsmallmatrix} C & 0 \\ A^\star & B^\star \end{bsmallmatrix}$ of full column rank.
	
	\begin{Prob}\label{prob.def}
		Consider the system \eqref{eq.sys-main} under Assumption~\ref{as.coeff}, and let $\ydes \in \R^\ny$ and $\eps \ge 0$ be a desired target and regulation precision, respectively. 
		\begin{enumerate}[label=(\roman*)]
			\item \label{prob.cont}
			{\bf Control synthesis:} Synthesize an output control {$y_{[0,t]} \mapsto u(t)$},\footnote{The notation $y_{[0,t]}$ is the restriction of the function~$y$ to the set $[0,t]$, that is, $\{y(s): s\in[0,t]\}$.}
			$t \ge 0$, in order to ensure the $\eps$-practical output regulation in the sense of \eqref{eq.terminal}. 
			
			\item \label{prob.dis}
			{\bf {Sampled-time event-based} emulation:} Given a prescribed series of measurement sampled-times, design a triggering mechanism to update the control along with a guaranteed precision of the desired output regulation~\eqref{eq.terminal}. 
			
		\end{enumerate}
	\end{Prob}
	
	We start with {designing} a {continuous-time controller (Section~\ref{sec:Continuous Robust Control}) whose sampled-time redesign, or emulation, is considered in Section~\ref{sec:Triggered}}. Note that the viability of the sampled-time emulation 
	reflects a certain robustness level of the continuous-time controller.

%===============================================================================
\section{Continuous-Time Control Design} 
\label{sec:Continuous Robust Control}
%===============================================================================
{The main focus of this section is Problem~\ref{prob.def}\ref{prob.cont}. We first find a structure of the controller ensuring that the closed-loop system has an equilibrium $(\xdes,\udes)$ such that $\ydes=C\xdes$, and then provide sufficient conditions \textcolor{black}{guaranteeing} that this equilibrium is globally asymptotically stable.} The existence of an equilibrium is natural, if one {is} interested in the $\eps$-practical stability~\eqref{eq.terminal} with an arbitrarily small~$\eps$.
	
A possible control architecture, and perhaps the simplest form, is {the static} controller $u(t) = \Dc y(t)+\eta$. {Unfortunately, to provide the existence of an equilibrium from Assumption~\ref{as.coeff}\ref{as.coeff.target}, the parameter $\eta=\udes-\Dc\ydes$ should depend on $\udes$, which, in turn, depends on the uncertain matrices $A^{\star}$ {and} $B^{\star}$ and function $k^\star$. For this reason, we propose a dynamic controller, being a multidimensional counterpart of the classical proportional-integral control}.

\subsection{Dynamic control and equilibrium existence}\label{subsec:equilibrium}
	
Consider {now a more general} \emph{dynamic} controller
	\begin{align}\label{eq.controller}
	\begin{dcases}
	\dot{\xc}(t)=\Ac \xc(t)+\Bc y(t)+\xi\\
	u(t)=\Cc \xc(t)+\Dc y(t)+\eta,
	\end{dcases}
	\end{align}
where matrices $\Ac,\Cc \in \R^{\enu \times \enu}$, $\Bc,\Dc \in \R^{\enu \times \ny}$ and $\xi, \eta \in \R^{\enu}$ are {the} design parameters. {These additional parameters in~\eqref{eq.controller} enable one to {make} the equilibrium $(x^*,w^*)$ of the closed-loop system~\eqref{eq.sys-main} and \eqref{eq.controller} {compatible} with the target value $\ydes$ {in the face of the parametric uncertainty~\eqref{eq.assumption1}.}}
	
\begin{Lem}[Closed-loop equilibrium]\label{lem.equilibrium}
	{If} Assumption~\ref{as.coeff}\ref{as.coeff.target} holds, the matrix $\Cc$ {has} full column rank, and {the controller parameters are such that}
	\be\label{eq.coeff}
	\Ac = 0 \quad \textrm{and} \quad \xi=-\Bc\ydes, \,\ee
	{then} the {closed-loop system~\eqref{eq.sys-main} and \eqref{eq.controller} has an}
	equilibrium~$(\xdes,\wdes)$, where $\xdes$ is introduced in Assumption~\ref{as.coeff}\ref{as.coeff.target}.
\end{Lem}
	
\begin{proof}
	Since the matrix~$\Cc$ has full {column} rank, there exists $\wdes\in\R^\enu$ such that $\Cc \wdes+\Dc\ydes+\eta = \udes$, where $\udes$ is \textcolor{black}{given by}~\eqref{eq.eq-des}. {In view of Assumption~\ref{as.coeff}\ref{as.coeff.target} and~\eqref{eq.coeff}}, {the point $(x^d,\xc^d)\in \R^{\nx+\enu}$ obeys the algebraic equations}
	\begin{align}\label{eq.equilibrium}
	\left\lbrace
	\begin{matrix*}[l]
    	A^{\star}\xdes+B^{\star}(\Cc \wdes+\Dc C\xdes+\eta)+k^\star(\xdes)=0, \vspace{1mm}\\
    	\Ac \wdes+\Bc C\xdes+\xi = \Bc(\ydes - C\xdes) = 0,
	\end{matrix*}
	\right.
	\end{align}
	\textcolor{black}{Hence, it is an equilibrium for the closed-loop system.}
\end{proof}
{Notice that}	
the {controller's} parameters~$\Bc, \Dc$, \textcolor{black}{and} $\eta$ {do not influence the \emph{existence} of an equilibrium compatible with the {desired output} $\ydes$.} 
While $\Bc$ and $\Dc$ may influence the stability of the transient behavior of the closed-loop system, the vector~$\eta$ does not {affect stability and only determines~$w^d$}. 
{Hence, without} loss of generality, we set $\eta=-\Dc\ydes$. {Combining this with~\eqref{eq.coeff} and the controller~\eqref{eq.controller} shapes into}
\begin{equation}\label{eq.controller1}
	\left\lbrace
	\begin{matrix*}[l]
	\dot \xc(t)=\Bc\big(y(t)-\ydes\big) \vspace{1mm}\\
	u(t)=\Cc \xc(t)+\Dc\big(y(t)-\ydes\big) \,.
	\end{matrix*}\right.
\end{equation}
Note that the dynamic controller~\eqref{eq.controller1} may be {considered} as a (multidimensional) extension of the conventional PI controller.
	
\subsection{Closed-loop stability of transient behavior}
{The goal of this section is to design the controller parameters~$\Bc, \Cc$, \textcolor{black}{and} $\Dc$ such that the
 equilibrium from Lemma~\ref{lem.equilibrium} is (practically) stable.} To this end, {we} introduce the augmented state vector {of the closed-loop system} \textcolor{black}{as}
\begin{equation}
\label{eq.newvariables}
	\xcl(t)\coloneqq \begin{bmatrix} x(t)-\xdes \\ \xc(t)-\wdes \end{bmatrix}.
\end{equation}
	
Based on the system dynamics in~\eqref{eq.sys-main} together with the controller~\eqref{eq.controller1}, \textcolor{black}{it is obtained that}
\begin{align}
	\label{eq.z_dynamic}
	\begin{array}{r}
	\dot{\xcl}=\left[\bar{A}+J^{\top}\Delta A J+(\bar{B}+J^{\top}\Delta B J)F\bar{C}\right]\xcl+J^\top(k^\star(J^\top \xcl)-k^{\star}(\xstar)),
	\end{array}
\end{align}
where $\Delta A = A^{\star}-A$ and $\Delta B = B^{\star}-B$ represent the uncertainty in the linear part of the system dynamics, and matrices $\bar{A},\bar{B},\bar{C}, F$, \textcolor{black}{and} $J$ are defined as \textcolor{black}{follows}.
\begin{align}
	\label{M_define}
	\begin{matrix*}[l]
	\bar{A} \Let \begin{bmatrix} A & 0 \\ 0& 0\end{bmatrix}, \, \bar{B} \Let \begin{bmatrix} B & 0 \\0&I \end{bmatrix},\,
	\bar{C} \Let \begin{bmatrix} C & 0 \\ 0 & I \end{bmatrix}, \, J\Let\begin{bmatrix} I_{\nx} & 0_{\nx\times \enu}\end{bmatrix}, \, F \Let \begin{bmatrix} \Dc & \Cc \\ \Bc & 0\end{bmatrix}
	\end{matrix*}
\end{align}
{It should be noted that} matrix~$F$ collects all the design variables of the controller.
{The goal of the controller design is to \textcolor{black}{guarantee} the (practical)} stability of the {system}~\eqref{eq.z_dynamic} for all uncertainties $\Delta A, \Delta B$, \textcolor{black}{and} $k^\star(\cdot)$ that meet Assumption~\ref{as.coeff}. Unfortunately, it turns out that the exact characterization of such an $F$ is provably intractable. In fact, \textcolor{black}{the special case of} checking the stability of the system~\eqref{eq.z_dynamic} for a given $F$ is also \textcolor{black}{a} difficult \textcolor{black}{problem}. This is formalized in the next proposition.
	{
\begin{Prop}[Intractability]
	\label{prop.interactability}
	Consider the system~\eqref{eq.sys-main} under Assumption~\ref{as.coeff}, and let the control signal follow the dynamics~\eqref{eq.controller1}. Then, \textcolor{black}{for a} given \textcolor{black}{set of the} control parameters (\textcolor{black}{i.e.,} matrix $F$ in \eqref{M_define}), \textcolor{black}{the problem of} checking whether the output target stability~\eqref{eq.terminal} holds for some $\eps \ge 0$
	is strongly NP hard and equivalent to
	\begin{align}\label{eq.forallthereexist}
    	\begin{matrix*}[l]
    	\forall \Delta A,\Delta B: |\Delta A| \leq \Mbdd, \, |\Delta B|\leq \Nbdd \,
    	\quad\exists P \in S^{\nx+\enu}_{\succ 0}:\\\Sym{P\left(\bar{A}+J^{\top}\Delta A J\right.\left.+(\bar{B}+J^{\top}\Delta B J)F\bar{C}\right)} \preceq 0.
    	\end{matrix*}
	\end{align}
\end{Prop}}
\begin{proof} Recall that the nonlinear term in the dynamics~\eqref{eq.z_dynamic} is uniformly bounded due to Assumption~\ref{as.coeff}\ref{as.coeff.nonlinear}. Therefore, thanks to the classical result of \cite[Theorem 9.1]{khalil2002nonlinear}, the stability of the system~\eqref{eq.z_dynamic} is equivalent to the stability of the linear part described as
	\begin{align}
	    \label{eq.z_dynamic_without}
	    \dot{\xcl}=\left[\bar{A}+J^{\top}\Delta A J+(\bar{B}+J^{\top}\Delta B J)F\bar{C}\right]\xcl \,.
	\end{align}
	From the classical linear system theory, we know that the stability of \eqref{eq.z_dynamic_without} is equivalent to the existence of a quadratic Lyapunov function~$V(\xcl)=\xcl^{\top}P\xcl$, where the symmetric positive definite matrix~$P$ may in general depend on the uncertainty in the dynamics. This assertion can be mathematically translated to checking whether the given controller parameter~$F$ satisfies~\eqref{eq.forallthereexist}. Note that the order of the quantifies implies that the matrix~$P$ may depend on the uncertain parameter $\Delta A$ and $\Delta B$. The assertion~\eqref{eq.forallthereexist} is indeed a special case of the problem of an interval matrix's stability~\cite{nemirovskii1993several}, which is proven to be strongly NP-hard~\cite[Corollary 2.6]{ahmadi2019complexity}.
\end{proof}
	
A useful technique to deal with the assertion similar to~\eqref{eq.forallthereexist} is to choose a so-called common Lyapunov function~\cite{ooba1997common}. Namely, we aim to find a positive-definite matrix~$P$ for all possible model parameters, i.e., the assertion~\eqref{eq.forallthereexist} is replaced with a more conservative requirement as follows:
\begin{align}\label{eq.thereexistforall}
    \begin{matrix*}[l]
    \exists P \in S^{\nx+\enu}_{\succ 0}\quad\forall \Delta A,\Delta B: |\Delta A| \leq \Mbdd, \, |\Delta B|\leq \Nbdd \\
    \Sym{P\left(\bar{A}+J^{\top}\Delta A J\right.\left.+(\bar{B}+J^{\top}\Delta B J)F\bar{C}\right)} \preceq 0.
    \end{matrix*}
\end{align}
Note that the only difference between \eqref{eq.forallthereexist} and the conservative assertion in \eqref{eq.thereexistforall} is the order of quantifiers between the Lyapunov matrix~$P$ and the linear dynamics uncertainties~$\Delta A$ and $\Delta B$. The argument~\eqref{eq.thereexistforall} is a special subclass of problems known as the ``matrix cube problems"~\cite{ben2002tractable}. While this class of problems is also provably hard~\cite[Proposition 4.1]{ben2002tractable}, the state-of-the-art in the convex optimization literature offers an attractive sufficient condition where the resulting conservatism is bounded {\em independently} of the size of the problem~\cite{ben2003extended}. Building on these developments, we will provide an optimization framework to design the controller \textcolor{black}{parameters} along with a corresponding common Lyapunov function.
	
\begin{Thm} [Robust control \& common Lyapunov function]	\label{thm.Continuous}
	Consider the system~\eqref{eq.sys-main}, {satisfying} Assumption~\ref{as.coeff}, and the {controller}~\eqref{eq.controller1}. \textcolor{black}{Also,} consider the optimization program
	\begin{align}\label{eq.continuous_NMI}
		\left\{
		\begin{matrix*}[l]
		\max \hspace{2mm}{\alpha\zeta^{-1}}\\\vspace{1mm}
		~{\rm s.t.} \hspace{3mm} \alpha \in \R, \quad {\zeta}, \kappa_{ij},\mu_{ik} \in \R_{>0},\quad
		 P \in \mathbb{S}^{\nx+\enu}_{\succ 0}, \text{\space}\Cc \in \R^{\enu \times \enu},\text{\space} \Bc,\Dc \in \R^{\enu \times \ny}
		\\\vspace{1mm}\hspace{9mm}F = \begin{bmatrix} \Dc & \Cc \\ \Bc & 0  \end{bmatrix}, \text{\space}M=\Sym{P\bar{A}+P\bar{B}F\bar{C}}+\alpha I
		\\\vspace{1mm}\hspace{9mm}G_1=\BlkDiag{-\kappa_{ij}a_{b_{ij}}^{-2}}_{i,j},\text{\space}G_2=\BlkDiag{-\mu_{ik}b_{b_{ik}}^{-2}}_{i,k},\,G_3=\BlkDiag{-\mu_{ik}^{-1}}_{i,k}
		\\\hspace{9mm}H_1=PJ^{\top}(\mathds{1}_\nx\otimes I_\nx),\,H_2=\bar{C}^{\top}F^{\top}J^{\top}\begin{bmatrix}\mathds{1}_\enu\otimes e_1 &\dots&\mathds{1}_\enu\otimes e_\nx\end{bmatrix}
		\\\vspace{1mm}\hspace{9mm}\begin{bmatrix}
		M+\sum_{i,j}\kappa_{ij}J^{\top}e_j^{\top}e_jJ&*&*&*&* \\
		H_1^{\top} &G_1&*&*&*\\
		H_1^{\top}&0&G_2&*&*\\
		H_2^{\top}&0&0&G_3&*\\
		JP&0&0&0& - {\zeta} I,
		\end{bmatrix} \preceq 0
		\end{matrix*}\right.
	\end{align}	
	where $\alpha_*, \zeta_*$ and $P_*$ denote the optimal solutions of corresponding decision variables. \textcolor{black}{If $\alpha_*>0$, then the controller {provides $\epsc$-practical output regulation~\eqref{eq.terminal} where}}
	\begin{align}\label{epsc}
		\epsc = \rhob\|\bar{C}\|\sqrt{\dfrac{\lambda_{\max}(P_*)}{\textcolor{black}{\alpha_*\zeta^{-1}_*}\lambda_{\min}(P_*)}} \,.
	\end{align}
	In particular, {if $\rhob = 0$ (\textcolor{black}{i.e.}, the nonlinear term vanishes to a constant)} and $\alpha_* > 0$, then the closed-loop system is exponentially stable and $\lim_{t \rightarrow \infty} y(t) = \ydes$.
\end{Thm}
\begin{proof}
	{Consider }a quadratic Lyapunov function $V(z)=\xcl^{\top}P\xcl$. {The \textcolor{black}{time-}derivative of $V$} along the trajectories of \eqref{eq.z_dynamic} {is} 
	\begin{align*}
		\dfrac{1}{2} \dfrac{d}{dt} V(\xcl) & = {\xcl^{\top}P\left(\bar{A}+\bar{B}F\bar{C}\right)\xcl}
		 + \, {\xcl^\top P\left( J^{\top}\Delta A J+J^{\top}\Delta B JF\bar{C}\right)\xcl}  + {\xcl^\top P J^\top (k^{\star}(J^\top \xcl)-k^{\star}(\xstar))},
	\end{align*}
	where the last term involving the nonlinear term can be estimated by invoking \textcolor{black}{the} Young's inequality as follows.
	\begin{align}
		\nonumber
		2{\xcl^\top P J^\top \big(k^{\star}(J^\top \xcl )-k^{\star}(\xstar)\big)}\leq \zeta^{-1} \xcl^\top P J^\top J P \xcl
		+ {\zeta}\big\|k^{\star}(J^\top \xcl)-k^{\star}(\xstar))\big\|^2 %\\
		\leq \, \zeta^{-1} \xcl^\top P J^\top J P \xcl+{\zeta} \rhob^2 \,.
	\end{align}
	Notice that the parameter~$\zeta \in \Rp$ is a positive scalar, and the last inequality is an immediate consequence of~\eqref{eq.assumption2}. In the light of the latter estimate, one can observe that if the inequality
	\begin{align}
		\label{eq.asymp}
		\begin{array}{r}
		\Big[P(\bar{A}+\bar{B}F\bar{C}) + P( J^{\top}\Delta A J+J^{\top}\Delta B JF\bar{C}) +  {\dfrac{\zeta^{-1}}{2}}P J^\top J P \Big]^\dagger\preceq -\alpha I,
		\end{array}
	\end{align}
	holds for some $\alpha \in \Rp$, then the dynamics of the Lyapunov function value along with system trajectories \textcolor{black}{satisfy}
	\begin{align}
		{\label{eq.bound}}
		\dfrac{1}{2} \dfrac{d}{dt} V(\xcl) \le -\alpha \|\xcl\|^2 + \zeta \rhob^2 \le {\dfrac{-\alpha} {\lambda_{\max}(P_*)}}V(\xcl)+ \zeta \rhob^2.
	\end{align}
	The above observation implies that $\limsup_{t \rightarrow \infty}V\big(\xcl(t)\big) \le \lambda_{\max}(P_*) \zeta \rhob^2/\alpha$, which together with the simple bound $\lambda_{\min}(P_*) \|\xcl\|^2 \le V(\xcl)$, leads to
	\begin{align*}
		\limsup_{t \rightarrow \infty} \|y(t) - \ydes\| &\le \limsup_{t \rightarrow \infty} \|\bar{C}\| \|\xcl(t)\| 
		\le \limsup_{t \rightarrow \infty} \|\bar{C}\| \sqrt{\dfrac{V\big(\xcl(t)\big)}{\lambda_{\min}(P_*)}}
		\le \epsc \,,
	\end{align*}
	where $\epsc$ is defined as in \eqref{epsc}.
	Hence, the above observation indicates that under the requirement~\eqref{eq.asymp} for some $\alpha > 0$, the desired assertion holds. {Next}, we aim to replace the robust inequality~\eqref{eq.asymp} by a more conservative criterion, which in turn can be verified efficiently. This procedure consists of several steps. Introducing the variable~$M\coloneqq \Sym{P\bar{A}+P\bar{B}F\bar{C}}+\alpha {I}$, the inequality~\eqref{eq.asymp} {is} rewritten as
	\begin{align}
		& -M-\zeta^{-1} PJ^\top JP+\Big[P J^{\top}\sum_{i=1}^{\nx}\Big(\sum_{j=1}^{\nx}(\delta a_{{ij}})e_i^{\top}e_j\Big) J 
		\label{Lyap_derivative_1}
		 + PJ^{\top}\sum_{i=1}^{\nx}\Big(\sum_{k=1}^{\enu}(\delta b_{{ik}})e_i^{\top}e_k\Big) JF\bar{C}\Big]^\dagger \succeq 0,
	\end{align}
	where the uncertainty parameters are described element-wise as $\Delta A = [\delta a_{ij}]$ and $\Delta B = [\delta b_{ij}]$. Recall that the condition~\eqref{Lyap_derivative_1} has to hold for all uncertain parameters, i.e., it is a robust constraint. Thanks to \cite[Theorem 3.1]{ben2003extended}, constraint~\eqref{Lyap_derivative_1} holds if there exist parameters $D_{ij}$, $E_{ik}$, $\lambda_{ij}$, $\gamma_{ik}$, where $i,j \in \{1,\dots,\nx\}$ and $k \in \{1,\dots,\enu\}$, \textcolor{black}{such} that
	\begin{align}
		\label{nemirovskii}
		\begin{matrix*}[l]
		\begin{bmatrix}
		D_{ij}-\lambda_{ij} a_{b_{ij}}^2 \xcl^{\top}PJ^{\top}e_i^{\top}e_iJP\xcl & *\\ e_jJ\xcl & \lambda_{ij} I
		\end{bmatrix} \succeq 0, \vspace{1mm}\\
		\begin{bmatrix}
		E_{ik}-\gamma_{ik} b_{b_{ik}}^2 \xcl^{\top}PJ^{\top}e_i^{\top}e_iJP\xcl & * \\ e_kJF\bar{C}{\xcl} & \gamma_{ik} I
		\end{bmatrix} \succeq 0,	\vspace{1mm} \\
		-\xcl^\top\left(M+\zeta^{-1} PJ^\top JP\right)\xcl \, \geq\, \sum_{i,j} D_{ij} +  \sum_{i,k} E_{ik}.
		\end{matrix*}	
	\end{align}
	By deploying the standard Schur complement in the first two inequalities of \eqref{nemirovskii}, we arrive at
	\begin{align}
		\label{nemirovskii+shur}
		\begin{matrix*}[l]
		\lambda_{ij},\gamma_{ik} > 0, \vspace{1mm}\\
		D_{ij}-\lambda_{ij} a_{b_{ij}}^2 \xcl^{\top}PJ^{\top}e_i^{\top}e_iJP\xcl - \lambda_{ij}^{-1}\xcl^{\top}J^{\top}e_j^{\top}e_jJ\xcl\geq 0, \vspace{1mm}\\
		E_{ik}-\gamma_{ik} b_{b_{ij}}^2 \xcl^{\top}PJ^{\top}e_i^{\top}e_iJP\xcl -\gamma_{ik}^{-1}{\xcl}^{\top}\bar{C}^{\top}F^{\top} J^{\top}e_k^{\top}e_kJF\bar{C}{\xcl}\geq 0,	 \vspace{1mm}\\
		-\xcl^\top\left(M+\zeta^{-1} PJ^\top JP\right)\xcl \, \geq\, \sum_{i,j} D_{ij} +  \sum_{i,k} E_{ik}.	
		\end{matrix*}
	\end{align}
	Eliminating $\{D_{ij}\}_{i,j}$ \textcolor{black}{and} $\{E_{ik}\}_{ i,k}$ and \textcolor{black}{doing some} straightforward computation\textcolor{black}{s}, the above inequalities reduces to
	\begin{align}
		\begin{matrix*}[l]
		\label{nemirovskii+shur_1}
		\lambda_{ij},\gamma_{ik} > 0, \vspace{1mm}\\
		M+\zeta^{-1} PJ^\top JP+\sum_{i,j}\kappa_{ij}J^{\top}e_j^{\top}e_jJ-H_1G_1^{-1} H_1^{\top}-H_1G_2^{-1} H_1^{\top}-H_2G_3^{-1} H_2^{\top} \preceq 0,
		\end{matrix*}
	\end{align}
	where the matrices $G_1, G_2, G_3, H_1$, \textcolor{black}{and} $H_2$ are defined as in \eqref{eq.continuous_NMI}. The proof is then concluded by applying yet again the Schur complement to the inequality~\eqref{nemirovskii+shur_1} and replace the variables $\kappa_{ij} = \lambda_{ij}^{-1}$ and $\mu_{ik} = \gamma_{ik}^{-1}$. 
	We note that since $\zeta > 0$, then $\alpha \ge 0$ if and only the objective function~$\alpha\zeta^{-1} \ge 0$.
	Therefore, the explicit positivity constraint over the variable~$\alpha$ can be discarded without any impact on the assertion of the theorem. In fact, the elimination of this constraint allows the program~\eqref{eq.continuous_NMI} being always feasible.
	Finally, we also note that the second part of the assertion is a straightforward consequence of the bound \eqref{epsc} and the fact that asymptotic stability and exponential stability in linear system coincide.
	\end{proof}
	
	The optimization program~\eqref{eq.continuous_NMI} in Theorem~\ref{thm.Continuous} is, in general, non-convex. We however highlight two important features of this program: (i) It is a tool enabling \emph{co-design} of a controller and \textcolor{black}{obtain} a Lyapunov function for the closed-loop system, and (ii) when the control parameters are fixed, the resulting program reduces to a linear matrix inequality (LMI), which is amenable to the off-the-shelves convex optimization solvers. The latter argument is formalized as follows.
	
\begin{Cor} [Controller certification via convex optimization]\label{cor.nominaldesign}
	Consider system~\eqref{eq.sys-main} satisfying Assumption~\ref{as.coeff} {that is closed through the feedback~\eqref{eq.controller1} with some fixed coefficients~\eqref{M_define}.} 
	Consider the optimization program	
\begin{align}
	\label{eq.tolerate}
	\left\{
	\begin{matrix*}[l]
	\max ~ \quad \alpha\zeta^{-1}\hspace{5mm}
	\\
	~{\rm s.t.} \hspace{6mm} \alpha \in \R, \text{\space}\zeta, \kappa_{ij} , \mu_{ik} \in \Rp,\text{\space} P \in \mathbb{S}_{\succ 0}^{\nx + \enu} \\
	\vspace{1mm}\hspace{12mm} M^\prime=M+\sum_{i,j}\kappa_{ij}J^{\top}e_j^{\top}e_jJ-H_2G_3^{-1}H_2^{\top}\\\vspace{1mm}
	\hspace{11mm} \begin{bmatrix}
	M^\prime&*&*&* \\
	H_1^{\top} &G_1&*&*\\
	H_1^{\top}&0&G_2&*\\
	JP&0&0&-\zeta I
	\end{bmatrix} \preceq 0
	\end{matrix*}\right.
\end{align}
where the matrices~$C,F,G_1,G_2,G_3,H_1$, \textcolor{black}{and} $H_2$ are {defined on the basis of }the system and control parameters\footnote{Formally speaking, the objective function in~\eqref{eq.tolerate} is not convex. However, since the only source of nonconvexity is the scalar variable~$\zeta$, a straightforward approach is to select this variable through a grid-search or bisection.}. Let $\alpha_*$, ${\zeta}_*$, and ${P}_*$ denote an optimizer of the program~\eqref{eq.tolerate}. Then, if $\alpha_* > 0$, then the output target control~\eqref{eq.terminal} is fulfilled for all $\eps \ge \epsc$ as defined in \eqref{eq.tolerate}. Moreover, if $\alpha_* \leq 0$, then there \textcolor{black}{exist} dynamics matrices $A^\star$ and $B^\star$ such that
\begin{align}
	\nonumber
	|A^\star-A|\leq \frac{\pi}{2}\Mbdd, \quad |B^\star-B|\leq\frac{\pi}{2}\Nbdd,
\end{align}
	and the closed-loop system is unstable.
\end{Cor}
	
\begin{proof}
	{Considering} the optimization program~\eqref{eq.continuous_NMI} {with fixed matrix $F$},
	the matrix~$H_2$ is also {fixed. The first statement is obtained by}
	applying the standard Schur complement as in~\eqref{nemirovskii+shur_1}. 
	{The second statement follows from}
	~\cite[Theorem 3.1]{ben2003extended} stating that the convex characterization of~\eqref{eq.thereexistforall} (i.e., the step from \eqref{Lyap_derivative_1} to \eqref{nemirovskii}) is tight up to multiplier~$\pi/2$.
\end{proof}
	
We close this section by a remark on the different sources of conservatism concerning the approach proposed in this section. {It is needless to say that any numerical progress at the frontier of each of these sources will lead to an improvement of the solution method in this article.}
	
\begin{Rem}[Conservatism of the proposed approach]\label{rem.conservatism}
	The path from the output target control~\eqref{eq.terminal} to the numerical solution of the optimization program~\eqref{eq.continuous_NMI} constitutes three steps that are only sufficient conditions and may contribute to the level of conservatism:
	(i) to restrict to a common Lyapunov function, i.e., the transition from \eqref{eq.forallthereexist} to 
	(ii) to apply the state-of-the-art matrix cube problem from \eqref{Lyap_derivative_1} to \eqref{nemirovskii}, and
	(iii) to numerically solve the finite, but possibly nonconvex, optimization program~\eqref{eq.continuous_NMI}.
	As detailed in Corollary~\ref{cor.nominaldesign}, the conservatism introduced by step (ii) is actually tight up to a constant independently of the dimension of the problem. {With regards to the nonconvexity issue raised in step (iii), we will examine a recent approximation technique proposed by~\cite{lee2018sequential} that is particularly tailored to deal with bilinearity of a similar kind in Theorem~\ref{thm.Continuous}; this will be reported in Section~\ref{Sec:Numerical}.}
\end{Rem}
%===========================================================
\section{{Aperiodic event-triggered robust control}} \label{sec:Triggered}
%===========================================================
{In this section, we address Problem~\ref{prob.def}\ref{prob.dis} aiming to synthesize a \emph{sampled-time} counterpart of the controller,} \textcolor{black}{which can access the system output~$y(\cdot)$ only at \emph{sampled} instants~$\{t_s\}_{s \in \N}$. The sequence~$t_s$ is \emph{predefined} by, for instance, an external message scheduler. Throughout this study we require that $t_s < t_{s+1}$ and $t_s$ tends to infinity when $s$ increases. The latter is a sufficient condition to ensure a ``{\em Zeno-free}" control design, a necessary requirement to avoid possible infinite switches in a finite-time period. We note that the inter-sampling intervals $t_{s+1}-t_s$ need not be constant, i.e., we allow an arbitrary \emph{aperiodic} time sampling. Continuous-time controller~\eqref{eq.controller1} is then naturally replaced by its sampled-time \emph{emulation} where the output signal~$y(t)$ fed to~\eqref{eq.controller1} within each interval $[t_s,t_{s+1})$ is replaced by its latest measurement $y(t_s)$:}
\begin{align}\label{eq.emulation-w}
	\xc(t) = \xc(t_s) + (t - t_s)\Bc(y(t_s)-\ydes), {t\in [t_s,t_{s+1})}
\end{align}
On the actuation side, the simplest scenario is to compute the new control input upon receiving measurement~$y(t_s)$, which remains constant till the next measurement~$y(t_{s+1})$ arrives:
\begin{equation}\label{eq.emulation-u}
	u(t) = \Cc\xc(t_s)+\Dc\big(y(t_s)-\ydes\big),\quad t\in [t_s,t_{s+1}).
\end{equation}
Note that $u(t)$ takes a constant value within the time interval~$t \in [t_s, t_{s+1})$. More generally, one may consider an \emph{event-triggered} strategy: Upon arrival of the new measurement $y(t_s)$, the control input is updated only if a triggering condition is fulfilled. {This criteria may reflect how far the plant's output or the controller's state have visibly changed since the last time that the control signal was updated.} 
	
\textcolor{black}{Formally, assume that the control input has been updated for the last time at $t=t_j$. Upon the arrival of the new measurement $y(t_s)$, where $t_s>t_j$, the \emph{triggering} condition is validated that involves the vector $v(t_j,t_s) \Let \big[\xc(t_j)^\top, y(t_j)^\top, \xc(t_s)^\top, y(t_s)^\top\big]^\top$.}
	
Inspired by~\cite{heemels2013periodic}, we consider a triggering condition as follows
\begin{equation}\label{eq.triggering}
	\begin{bmatrix} v(t_j,t_s) \\ 1 \end{bmatrix}^\top \QV \begin{bmatrix} v(t_j,t_s) \\ 1 \end{bmatrix} \geq 0.
\end{equation}
The condition~\eqref{eq.triggering} is slightly more generalized than the one proposed in~\cite{heemels2013periodic} in a way that it also supports constant thresholds. Note that the information vector~$v(t_j,t_s)$ is augmented by a constant $1$. If~\eqref{eq.triggering} holds, the control input is updated: we set $j=s$ and find $u(t_j)=u(t_s)$ from~\eqref{eq.emulation-u}. 
In \textcolor{black}{the} case \textcolor{black}{that} \eqref{eq.triggering} does not hold, the control \textcolor{black}{input} remains unchanged till {at least} time~$t_{s+1}$. {This procedure is summarized in Algorithm~\ref{alg.trigger}.
\begin{algorithm}
	\caption{Aperiodic Event-Triggered Control~(AETC)}
	\label{alg.trigger}
	\begin{algorithmic}[1]
		\\ \textbf{Initialization}: \textcolor{black}{Consider} sample instants $\{t_s\}_{s \in \N}$, initial measurement~$y_0$, \textcolor{black}{and} initial control state $\xc_0=0$. Set $j=0$, compute $u_0$ from~\eqref{eq.emulation-u}, and send it to the system~\eqref{eq.sys-main}.
		\\ \textbf{Upon receiving $y(t_s)$}, find $\xc (t_s)$ from \eqref{eq.emulation-w}.
		\begin{itemize}
			\item If~\eqref{eq.triggering} holds, then set $j \leftarrow s$, compute $u(t_j)=u(t_s)$ from~\eqref{eq.emulation-u} and send it to the system~\eqref{eq.sys-main};
			\item otherwise, keep~$u(t_s) = u(t_j)$ for $t\in [t_s,t_{s+1})$, i.e., nothing is required to be communicated to~\eqref{eq.sys-main}.
		\end{itemize}
		\\ \textbf{Set } $s \leftarrow s+1$ and go to step 2.
	\end{algorithmic}
\end{algorithm}}
	
\begin{Rem}[Special triggering mechanisms]\label{rem:special}
	If in \eqref{eq.triggering} $\QV = 0$, the {control strategy} reduces to the usual aperiodic sampled-time (or digital) control. As pointed out in~\cite{heemels2013periodic}, the quadratic form~\eqref{eq.triggering} subsumes the relative event-triggered mechanism~\cite{tabuada2007event}. The mechanism~\eqref{eq.triggering} includes
	the absolute event-triggered mechanism~\cite{zhou2019periodic} and mixed event-triggered mechanism~\cite{donkers2012output} \textcolor{black}{as its special cases}. More specifically, when
	\begin{align}\label{Q0}
	\QV=\Qtau(\qt_0,\qt_1) \Let
	\left[\begin{matrix}
	I & * &* & * & * \\
	0 & I & * & * & * \\
	-I & 0 & I-\qt_1I &* & * \\
	0 & -I & 0 & I-\qt_1I & * \\
	0 & 0 &0 &0 &-\qt_0
	\end{matrix}\right],
	\end{align}
	the triggering mechanism~\eqref{eq.triggering} is translated into the condition
	\begin{align}\label{eq.etc-structure}
	\left\|\begin{bmatrix} \xc(t_s)-\xc(t_j) \\ y(t_s)- y(t_j) \end{bmatrix}\right\|^2 \ge \qt_0 + \qt_1\left\|\begin{bmatrix} \xc(t_s) \\ y(t_s) \end{bmatrix}\right\|^2 \,.
	\end{align}
\end{Rem}
In summary, the aperiodic event-triggered control (AETC) mechanism introduced above entails two key components: the time instants $\{t_s\}_{s \in \N}$, and the triggering mechanism~\eqref{eq.triggering} characterized by the matrix~$\QV$. \textcolor{black}{By definition, we know that $t_s\to\infty$, and as such, all solutions of the closed-loop system are forward complete, i.e., no {\em Zeno} trajectories may exist.} In the rest of this section, we {analyze the sampled-time event-triggered emulation of the dynamic controller from} Section~\ref{sec:Continuous Robust Control} {and provide sufficient conditions ensuring~\eqref{eq.terminal}}.

Let us fix the controller parameters to a feasible solution~$({\Bc}_*,{\Cc}_*,{\Dc}_*)$ of the optimization program~\eqref{eq.continuous_NMI} along with the Lyapunov matrix~$P_*$. For the brevity of the exposition, we also introduce the following notation:
\begin{align} 
    \label{theta_MN}
    \begin{gathered}
    	\hat{F}_* \Let \begin{bmatrix} {\Dc}_* & {\Cc}_* \\ 0 & 0  \end{bmatrix},
        \quad \beta \coloneqq \|P_*\|\|{\Bc}_*\bar{C}\|, \quad
        \varrho_B \Let \left(\|\bar{B}\|+\|\Nbdd\|\right)^2\|\hat{F}_*\|^2,\\
        \varrho_{AB} \Let \varrho_B\|\bar{C}\|^2+\left(\|\bar{A}\|+\|\Mbdd\|\right)^2,\quad
        \vartheta_B \Let \max\limits_{|\Delta B|\leq  \Nbdd}\|P_*(\bar{B}+J^\top\Delta BJ)\hat{F}_*\|,\\
     	\vartheta_{AB} \Let \max\limits_{|\Delta A|\leq \Mbdd, \, |\Delta B|\leq \Nbdd}\|\bar{A}+J^{\top}{\Delta A}J +(\bar{B}+J^\top\Delta BJ-I)(F_*-\hat{F}_*)\|, \\ 	\mathfrak{e}(h) \Let \vartheta_{AB}^{-1}(e^{\vartheta_{AB} h}-1)
    \end{gathered}
\end{align}
Now we want to proceed with the main result of this section.
\begin{Thm}[Certified robust regulation under AETC]
	\label{thm.aperiodic}
	Consider the system~\eqref{eq.sys-main} obeying Assumption~\ref{as.coeff}. Let the matrices~$({\Bc}_*,{\Cc}_*,{\Dc}_*,P_*,\alpha_*,\zeta_*)$ be a feasible solution to optimization problem~\eqref{eq.continuous_NMI} where $\alpha_* > 0$. Consider the AETC {in Algorithm~\ref{alg.trigger}}, where the sequence $\{t_s\}_{s\in\N}$ and matrix $\QV$ are such that
	\[\hb\Let\sup_{s \in \N} (t_{s+1}-t_s) \le h_{\max} \quad \text{and} \quad \QV \preceq \Qtau(\qt_0,\qt_1).
	\]
	Here $\Qtau(\qt_0,\qt_1)$ \textcolor{black}{is given by}~\eqref{Q0} with some constants~$\qt_0, \qt_1 \ge 0$ and
	\begin{align}
    \label{eq.hmax_definition}
        h_{\max} \coloneqq \vartheta_{AB}^{-1}\ln\left(1+\vartheta_{AB}\sqrt{\dfrac{\alpha_*^2{\sqrt{\qt_1}\lambda_{\min}(P_*)}[(1+2\sqrt{\qt_1})^2\lambda_{\max}(P_*)]^{-1}-2\vartheta_B^2\qt_1\|\bar{C}\|^2}{6\vartheta_B^2(\qt_1\varrho_B\|\bar{C}\|^4+6\varrho_{AB}\|\bar{C}\|^2)+3\beta^2(\varrho_B\qt_1\|\bar{C}\|^2+\varrho_{AB})^2}}\right).
    \end{align}
	Then, the closed-loop system under AETC is $\epsd$-practical {output} stable in the sense of~\eqref{eq.terminal} where
	\begin{equation}
	\label{eq.eps_d.def}
	\epsd^2=\mathfrak{f}_1(\hb,{\qt_1})\qt_0+\mathfrak{f}_2(\hb,\qt_1)\rhob^2,
	\end{equation}
	in which the constants~$\mathfrak{f}_1$ \textcolor{black}{and} $\mathfrak{f}_2$ {can be explicitly expressed in form~\eqref{eq.f.def}, depending only on $\hb, \qt_1, P_\star, \bar{C}$, and parameters~\eqref{theta_MN}.
	}
	\scriptsize
    {
    \begin{subequations}
        \label{eq.f.def}
            \begin{align}
                \label{eq.f1.def}
			    \mathfrak{f}_1\left(\hb,\qt_1\right)&\coloneqq \dfrac{\vartheta_B^2\left(2+6\varrho_B\|\bar{C}\|^2\mathfrak{e}^2(\hb)\right)\|\bar{C}\|^4+3\beta^2\varrho_B\|\bar{C}\|^4\mathfrak{e}^2(\hb)}{-\vartheta_B^2\left(2\qt_1\|\bar{C}\|^2+6\qt_1\varrho_B\|\bar{C}\|^4\mathfrak{e}^2(\hb)+6\varrho_{AB}\|\bar{C}\|^2\mathfrak{e}^2(\hb)\right)-3\beta^2(\varrho_B\qt_1\|\bar{C}\|^2+\varrho_{AB})^2\mathfrak{e}^2(\hb)+\alpha_*^2\dfrac{\sqrt{\qt_1}\lambda_{\min}(P_*)}{(1+2\sqrt{\qt_1})^2\lambda_{\max}(P_*)}},\\
                \label{eq.f2.def}
			    \mathfrak{f}_2(\hb,\qt_1)&\coloneqq \dfrac{6\vartheta_B^2\|\bar{C}\|^6\mathfrak{e}^2(\hb)+3\beta^2\|\bar{C}\|^4\mathfrak{e}^2(\hb)+\alpha_*\zeta_*\|\bar{C}\|^2\sqrt{\qt_1}(1+2\sqrt{\qt_1})^{-1}}{-\vartheta_B^2\left(2\qt_1\|\bar{C}\|^2+6\qt_1\varrho_B\|\bar{C}\|^4\mathfrak{e}^2(\hb)+6\varrho_{AB}\|\bar{C}\|^2\mathfrak{e}^2(\hb)\right)-3\beta^2(\varrho_B\qt_1\|\bar{C}\|^2+\varrho_{AB})^2\mathfrak{e}^2(\hb)+\alpha_*^2\dfrac{\sqrt{\qt_1}\lambda_{\min}(P_*)}{(1+2\sqrt{\qt_1})^2\lambda_{\max}(P_*)}}.
            \end{align}
        \end{subequations}
    } \normalsize
\end{Thm}
	
\begin{proof}
	Suppose $t\in [t_s,t_{s+1})$ and let $t_j\leq t_s$ be the last time instant when the control input was computed. Let $\xcl(t)$ be the state of the closed system defined in~\eqref{eq.newvariables}, and denote 
	\begin{equation*}
    	e(t)\coloneqq \begin{bmatrix}y(t_j)-y(t) \\ \xc(t_j)-\xc(t)\end{bmatrix}=\bar C(\xcl(t_j)-\xcl(t)),\,\bar {\xcl}(t):=\xcl(t)-\xcl(t_s).
	\end{equation*}
	where the matrix~$\bar{C}$ is defined in~\eqref{M_define}. Since~\eqref{eq.emulation-w} holds and $u(t)\equiv u(t_j)$ for $t\in[t_s,t_{s+1}]$, the closed-loop system's state evolves as 
    \begin{align}
        \label{eq.z_dynamic.dis}
        \dot{\xcl}(t) &=\left[\bar{A}+J^{\top}\Delta A  J + (\bar{B}+J^{\top}\Delta BJ){F_*}\bar{C}\right]\xcl(t) \\ \nonumber
        & \qquad + J^\top\left(k^\star\left(J^\top \xcl(t)\right)-k^{\star}(\xdes)\right)+(\hat{F}_*-F_*)\bar{C}{\bar {\xcl}(t)}   + (\bar{B}+J^{\top}\Delta BJ)\hat{F}_*e(t), \qquad
        t \in[t_s, t_{s+1}),
    \end{align}
	where the matrices~$\bar{A},\bar{B}$, {$J$} are defined in \eqref{M_define}. Consider the same Lyapunov function as in the continuous-time case $V(\xcl)=\xcl^{\top}P_*\xcl$
	whose time derivative along a trajectory of~\eqref{eq.z_dynamic.dis} can be computed by
    \begin{align}
         \label{lyap_der_p}
         \dfrac{1}{2}\dfrac{d}{dt}V(\xcl) & = \xcl^{\top}(t)P_*\Big((\bar{B}+J^{\top}\Delta BJ)\hat{F}_*e(t) \\ \nonumber
         & \quad  + \big(\bar{A}+J^{\top}\Delta A J+(\bar{B}+J^{\top}\Delta BJ){F_*}\bar{C}\big)\xcl(t) + (\hat{F}_*-F_*)\bar{C}{\bar {\xcl}(t)}+J^\top \big(k^{\star}(J^\top \xcl)-k^{\star}(\xdes)\big)\Big)\,.
    \end{align}	
	By assumption, we know that the objective function of the program~\eqref{eq.continuous_NMI} is positive, i.e., $\alpha_*\zeta^{-1}_*>0$. Due to Young's inequality,
	\begin{align} 
    	\nonumber
    	\begin{gathered}
    	2\xcl^{\top}(t)P_*\left(\bar{B}+J^\top\Delta BJ\right)\hat{F}_*e(t) 
    	\leq \psi_1\vartheta_B^2 \|\xcl(t)\|^2 + \psi_1^{-1}\|e(t)\|^2, \\
    	2\xcl^{\top}(t)P_*\big(\hat{F}_*-F_*\big)\bar{C}\bar{\xcl}(t)  
    	\leq \psi_2\beta^2 \|\xcl(t)\|^2+\psi_2^{-1}\|\bar{\xcl}(t)\|^2,
    	\end{gathered}
	\end{align}	
	where $\psi_1, \psi_2$ are two positive scalars to be specified later. Thus, the derivative $\dot V$ from~\eqref{lyap_der_p} can be estimated by
	\begin{align} 
	\label{lyap_der_p_1}
    	\dfrac{d}{dt} V(\xcl(t)) & \leq -(\alpha_*-\psi_1 \vartheta_B^2-\psi_2 \beta^2) \|\xcl(t)\|^2 + \zeta_*\rhob^2 +{\psi_1}^{-1} {\|e(t)\|^2}+{\psi_2}^{-1}{\|\bar {\xcl}(t)\|^2}.
	\end{align}
	One may also notice that since $\dot{\bar {\xcl}}(t)=\dot {\xcl}(t)$ and $e(t)=\bar{C}(\xcl(t_j)-\xcl(t_s))-\bar{C}\bar {\xcl}(t)$, the equation~\eqref{eq.z_dynamic.dis} is rewritten as
    \begin{align}
     \label{eq.zbar_dot}
    	&\dot{\bar{\xcl}}(t)   =\left[\bar{A}+J^{\top}\Delta A J+(\bar{B}+J^{\top}\Delta B J)F_*\bar{C}\right]\xcl(t_s)  + J^\top(k^\star(J^\top \xcl)-k^{\star}(\xdes)) \\ \nonumber
    	&  + (\bar{B}+J^{\top}\Delta BJ)\hat{F}_*\bar{C}(\xcl(t_j)-\xcl(t_s)) + \left[\bar{A}+J^{\top}{\Delta A}J+(\bar{B}+J^\top\Delta BJ-I)(F_*-\hat{F}_*)\right]\bar{C} \bar{\xcl}(t).
    \end{align}
	Recall that we have assumed $\hb \le h_{\max}$. Leveraging similar techniques as in \cite[Lemma 3]{kishida2017combined}, the solution of~\eqref{eq.zbar_dot} is estimated as
	\begin{align}
	\label{zbar_1}
    	& \|\bar{\xcl}(t)\| \leq  \Big[\left(\|\bar{B}\|+\|\Nbdd\|\right)\|\hat{F}_*\|\left\|e(t_s)\right\|+\rhob +\left(\|\bar{A}\|+\|\Mbdd\|+(\|\bar{B}\|+\|\Nbdd\|)\|{F}_*\bar{C}\|\right)\|\xcl(t_s)\|\Big]\mathfrak{e}(\hb)
	\end{align}
	where the constant~$\mathfrak{e}(h)$ is defined in~\eqref{theta_MN}. 
	Notice now that if $\QV \preceq \Qtau(\qt_0,\qt_1)$, we can conclude that $\|e(t_s)\|^2\leq \qt_0+\qt_1 \|\bar{C}\|^2 \|\xcl(t_s)\|^2$. This inequality automatically holds if $t_s=t_j$ (and $e(t_s)=0$). Otherwise, the triggering condition~\eqref{eq.triggering} is violated, whence
	\begin{align}
	\label{zbar_2}
	\|e(t)\|^2\leq \left(\|e(t_s)\|+\|e(t)-e(t_s)\|\right)^2 \leq 2\qt_0+2\qt_1\|\bar{C}\|^2\|\xcl(t_s)\|^2 + 2\|\bar{C}\|^2\|\bar{\xcl}(t)\|^2
	\end{align}
	for $t\in [t_s,t_{s+1}]$. Denote 
	\begin{align}
	\label{psi_values}
	\psi_1\coloneqq\sigma \vartheta_B^{-2}\alpha_*, \; \psi_2\coloneqq\sigma\beta^{-2}\alpha_*,\; \sigma\coloneqq {\sqrt{\qt_1}}({1+2\sqrt{\qt_1}})^{-1}.
	\end{align} 
	Equations \eqref{lyap_der_p_1} together with \eqref{zbar_1}-\eqref{psi_values} lead to
	\begin{align}
    	\label{lyap_der_p_2}
    	\dot{V}(\xcl(t)) \leq -\alpha_*(1-2\sigma) \|\xcl\|^2+\mathfrak{g}_1\|\xcl(t_s)\|^2+\mathfrak{g}_2,
	\end{align}
	where the constants $\mathfrak{g}_1,\mathfrak{g}_2$ are defined as
    \begin{subequations}			    \label{eq.g_i_definition}
        \begin{align}
         \mathfrak{g}_1 & = \sigma_1^{-1} \vartheta_B^2\alpha_*^{-1} \Big(2\qt_1\|\bar{C}\|^2 + 6 \qt_1\varrho_B \|\bar{C}\|^4\mathfrak{e}^2(\hb) + 6 \varrho_{AB} \|\bar{C}\|^2\mathfrak{e}^2(\hb)\Big) \\
         & \hspace{5.5cm} + 3 \sigma_2^{-1} \beta^2\alpha_*^{-1} (\varrho_B\qt_1\|\bar{C}\|^2+\varrho_{AB})^2\mathfrak{e}^2(\hb), \nonumber \\
         \mathfrak{g}_2 & =  \sigma_1^{-1}\vartheta_B^2\alpha_*^{-1}\Big(2\qt_0 + 6\qt_0\varrho_B \|\bar{C}\|^2 \mathfrak{e}^2(\hb) + 6 \|\bar{C}\|^2 \mathfrak{e}^2(\hb)\rhob^2\Big) \\
         & \hspace{5.5cm} + 3\sigma_2^{-1} \beta^2\alpha_*^{-1} \left(\varrho_B\qt_0+\rhob^2\right)\mathfrak{e}^2(\hb)+\zeta_*\rhob^2. \nonumber
        \end{align}
    \end{subequations}
	\textcolor{black}{Recalling that $V(z)\leq \|z\|^2\lambda_{max}(P_*)$ and denoting $h_s\coloneqq t_{s+1}-t_s$ and $\mathfrak{g}_3\coloneqq -\alpha_*(1-2\sigma)$, the inequality~\eqref{lyap_der_p_2} entails that}
	\begin{align}
    	\nonumber
    	V\left(t_{s+1}\right) \leq \left(e^{\mathfrak{g}_3\lambda_{\max}^{-1}(P_*)h_s}-1\right)\mathfrak{g}^{-1}_3\mathfrak{g}_2+ \left[e^{\mathfrak{g}_3\lambda_{\max}^{-1}(P_*)h_s}+\left(e^{\mathfrak{g}_3\lambda_{\max}^{-1}(P_*)h_s}-1\right)\mathfrak{g}^{-1}_3\mathfrak{g}_1\dfrac{\lambda_{\max}(P_*)}{\lambda_{\min}(P_*)}\right]V\left(t_{s}\right).
	\end{align}
	It can be shown  that the expression in brackets $[...]$ is less than $1$ if $h_s \leq \hb < h_{\max}$. 
	\textcolor{black}{Furthermore, if $\hb < h_{\max}$, then}
	\begin{align*}
    	\varlimsup_{t \rightarrow \infty} \|y(t)\|^2 & \leq \|\bar{C}\|^2\varlimsup_{t \rightarrow \infty}\|\xcl(t)\|^2  \leq \|\bar{C}\|^2\lambda^{-1}_{\min}(P_*)\varlimsup_{t \rightarrow \infty}V(t) \leq \|\bar{C}\|^2\dfrac{\mathfrak{g}_2\lambda_{\max}(P_*)}{-\mathfrak{g}_1\lambda_{\max}(P_*)-\mathfrak{g}_3\lambda_{\min}(P_*)}=\epsd^2.
	\end{align*}
	This implies that the system~\eqref{eq.sys-main} is $\epsd$-practical stable and also $y(t)$ converges to a ball with center $\ydes$ and radius $\epsd$.
\end{proof}
	
\begin{Rem}[Explicit inter-sampling bound]\label{rem:comparison}
	Theorem~\ref{thm.aperiodic} offers an AETC with a more general framework including absolute and relative thresholds whose maximal inter-sampling time~$h_{\max}$ can be found from~\eqref{eq.hmax_definition} (cf., \cite[Assumption~III.1]{heemels2013periodic}).
\end{Rem}
	
The setting in Theorem~\ref{thm.aperiodic} is clearly more stringent than the continuous measurements and actuation framework in Theorem~\ref{thm.Continuous}. Therefore, it is no longer surprising that the corresponding practical stability levels in \eqref{epsc} and \eqref{eq.eps_d.def} satisfy $\epsc \le \epsd$. The latter is essentially quantified based on three parameters: maximum inter-sampling bound~$h_{\max}$, and the absolute and relative triggering thresholds~$\qt_0$ \textcolor{black}{and} $\qt_1$ (cf. Remark~\ref{rem:special}). When $h_{\max}$ tends to $0$, our setting effectively moves from the aperiodic sampled measurement framework to the continuous domain, and when the thresholds~$\qt_0$ and $\qt_1$ tend to $0$, the event-triggered control mechanism transfers to the continuous-time implementation. {It can be shown that the gap between $\epsc$ and $\epsd$ in this case vanishes.}
\begin{Rem}[From discrete to continuous implementation]
	Let $\epsc$ be defined as in~\eqref{epsc} and $\epsd(\hb, \qt_0,\qt_1)$ in~\eqref{eq.eps_d.def} as a function of the relevant parameters $\hb, \qt_0$, \textcolor{black}{and} $\qt_1$. With a straightforward computation, one can inspect that
	\begin{align*}
	    \lim\limits_{\qt_0, \qt_1 \to 0} \lim_{\hb \to 0}\epsd(\hb, \qt_0,\qt_1) = \epsc.
	\end{align*}
\end{Rem}
{We} note that the practical stability certificate~$\epsd$ of the proposed AETC in~\eqref{eq.eps_d.def} may take $0$ values when~$\rhob = \qt_0 = 0$. This implies that even if the system is uncertain and we have an AETC in place, we may still be able to steer the output of the system to the desired target~$\ydes$. This interesting outcome, however, comes at the price of a bound on the absolute threshold~$\qt_1$. We close this section with the following result in this regard.
	
\begin{Cor}[Relative AETC threshold for perfect tracking]
	\label{cor.special}
	Suppose \textcolor{black}{that} the system~\eqref{eq.sys-main} is linear (i.e., $\rhob=0$ in Assumption~\ref{as.coeff}\ref{as.coeff.nonlinear}), the program~\eqref{eq.continuous_NMI} is feasible with~$\alpha_* > 0$, and the absolute threshold in Theorem~\ref{thm.aperiodic} is~$\qt_0=0$. {If}
	\begin{align*}
	    {\sqrt{\qt_1}}{(2\sqrt{\qt_1}+1)^2} < \dfrac{\alpha_*^2\lambda_{\min}(P_*)}{2\|\bar{L}\|^2\vartheta_B^2\lambda_{\max}(P_*)},
	\end{align*}
	then the regulation performance in~\eqref{eq.eps_d.def} is $\epsd=0$, i.e., the controller~\eqref{eq.controller1} implemented via the AETC scheme {in~Algorithm~\ref{alg.trigger}} steers the output of the system to the desired target~$\ydes$.
\end{Cor}
	
{\begin{proof}
		The proof is an immediate consequence of Theorem~\ref{thm.aperiodic}. It only suffices to check for which values of $\qt_1$ the maximal inter-sampling~$h_{\max}$ in~\eqref{eq.hmax_definition} is still well-defined.
\end{proof}

% ========================================================
\section{Numerical Method and Examples} 
\label{Sec:Numerical}
% ========================================================
Since optimization problem~\eqref{eq.continuous_NMI} is non-convex, special numerical techniques are introduced in the first part of this section and then utilized in the following to validate the main results of this study.

{\subsection{Numerical Method}
There are two types of nonlinearities in the optimization problem \eqref{eq.continuous_NMI}. The first type of these nonlinearities comes from cross products of decision variables and the second one comes from the appearance of inverse of some of decision variables. Since no general-purpose scheme is available to deal with bilinear matrix inequalities, one needs to resort to approximation approaches. {The paper~\cite[Section 4]{sadabadi2016static} has well reviewed several methods that can be used to deal with the bilinearities. Methods such as ``D-K iteration", ``Path-following", ``Linearized convex-concave decomposition", ``Riccati related approach" and ``Dual iteration approach" are examples of the methods mentioned in this article. In this paper, we will examine a recent powerful technique called ``{\em sequential parametric convex approximation}" from~\cite{lee2018sequential}, that is particularly tailored to deal with bilinearity of a similar kind in Theorem~\ref{thm.Continuous}. The main advantage of the sequential parametric convex approximation method is that it offers a simultaneous method for dealing with bilinearities and appearnce of inverse of some parameters, and also, it allows during iterations to optimize simultaneously over the control gain and the Lyapunov matrices. In addition, it offers better convergence speeds than algorithms such as linearized convex-concave decomposition.} 
We first provide two preparatory lemmas. 
\begin{Lem}
	\label{lem.cauchy}
	Let $\mathcal{Y}$ and $\mathcal{Z}$ be two matrices with appropriate dimensions. The inequality~$\Sym{\mathcal{Y}^{\top}\mathcal{Z}} \preceq 0$ holds if
	\begin{align}
	\label{eq.removeproduct}
	\begin{bmatrix}
	\Sym{(\mathcal{Y}-\mathcal{Y}_k)^{\top}\mathcal{Z}_k+\mathcal{Y}_k^{\top}(\mathcal{Z}-\mathcal{Z}_k)+\mathcal{Y}_k^{\top}\mathcal{Z}_k}	& * & * \\
	(\mathcal{Y}-\mathcal{Y}_k)^{\top} & -\mathcal{U} & * \\
	(\mathcal{Z}-\mathcal{Z}_k)^{\top} & 0 & -\mathcal{U}^{-1} \\
	\end{bmatrix} \preceq 0
	\end{align}
	where $\mathcal{Y}_k$ and $\mathcal{Z}_k$ are given matrices with the same size as $\mathcal Y$ and $\mathcal Z$, respectively, and $\mathcal{U} \in \mathbb{S}_{\succ0}$ is an arbitrary matrix.
\end{Lem}
Lemma~\ref{lem.cauchy} is essentially a combination of standard Young's inequality and Schur complement. It is worth noting that applying Young's inequality to the term $\Sym{\mathcal{Y}^{\top}\mathcal{Z}} \preceq 0$ yields an alternative approximation in the form of $\mathcal{Z}^{\top}\mathcal{U}^{-1}\mathcal{Z}+\mathcal{Y}^{\top}\mathcal{U}\mathcal{Y} \preceq 0$. However, if the constant matrices~$\mathcal{Y}_k$ and $\mathcal{Z}_k$ are close estimates of the variables~$\mathcal{Y}$ and $\mathcal{Z}$, respectively, then the proposed approximation in~\eqref{eq.removeproduct} is more efficient. We also note that in a context of optimization problem, the matrix~$\mathcal{U}$ is a degree of freedom, and that can be viewed as an additional decision variable. \newline
The next lemma suggests an idea to deal with the inverse of a decision variable in an optimization problem by introducing a linear over-approximation for the inverse of a matrix.
\begin{Lem} \cite[Lemma 2]{lee2016sequential}
	\label{lem.inverse}
	If $\mathcal{U},\mathcal{U}_k \in \mathbb{S}_{\succ 0}^{n}$, then
	\begin{align}
	\nonumber
	-\mathcal{U}^{-1} \preceq -2\mathcal{U}_k+\mathcal{U}_k^{-1}\mathcal{U}\mathcal{U}_k^{-1}.
	\end{align}
\end{Lem}
By some straightforward computations and using the results of Lemmas~\ref{lem.cauchy} and~\ref{lem.inverse}, one can observe that
\begin{align}
    \nonumber
    \begin{bmatrix}
    M_k+\sum_{i,j}\kappa_{ij}J^{\top}e_j^{\top}e_jJ&\begin{matrix}*&*&*\end{matrix} \\
    \begin{matrix} H_{1}^{\top} \\H_{1}^{\top}\\ X_k^{\top} \end{matrix} & G_k
    \end{bmatrix} \preceq 0
    \Rightarrow\begin{bmatrix}
    M+\sum_{i,j}\kappa_{ij}J^{\top}e_j^{\top}e_jJ&*&*&*&* \\
    H_1^{\top} &G_1&*&*&*\\
    H_1^{\top}&0&G_2&*&*\\
    H_2^{\top}&0&0&G_3&*\\
    JP&0&0&0&-\zeta I
    \end{bmatrix} \preceq 0
\end{align}
where,
\begin{align}
    \nonumber
    \begin{matrix*}[l]
    M_k\coloneqq [P\bar{A}+P_k\bar{B}(F-F_k)\bar{C}+P\bar{B}F_k\bar{C}]^\dagger+\alpha I\qquad
    G_{3_k}\coloneqq \BlkDiag{(-2\mu_{ij_k}+\mu_{ij})}_{i,j} 
    \\G_k\coloneqq \BlkDiag{G_1,G_{2},G_{3_k},-2U_k+U,-U,-\zeta I}
    \qquad H_{2_k}\coloneqq \BlkDiag{\mu_{ij_k}}H_2
    \\X_k\coloneqq\begin{bmatrix}H_{2_k}&(P-P_k)U_k^{\top}& \bar{B}(F-F_k)\bar{C} &PJ^\top \end{bmatrix}
    \end{matrix*}
\end{align}
Building on the above definitions, Algorithm \ref{alg:sequ}, as a sequential approximate algorithm, can be proposed to find a stationary point for the optimization problem \eqref{eq.continuous_NMI} (From \cite[Proposition 3]{lee2018sequential}, it can be proved that Algorithm \ref{alg:sequ} converges to a stationary point of \eqref{eq.continuous_NMI}).}
{{\begin{algorithm}
	\caption{Sequential Parametric Convex Approximation}
	\label{alg:sequ}
	\begin{algorithmic}[1]
		\\ \textbf{Set }$k=0,F_k=0,\mu_{ij_k}=1$.
		\\ \textbf{Solve } $\left\{\begin{matrix*}[l](\alpha_k,\zeta_k)=\arg\!\max \text{\space}(\ln{\alpha}-\ln{\zeta})
		\\\ensuremath{\,\textrm{Subject to}}
		\\\hspace{3mm}\lambda_{ij} \in \R_{>0}, W_{ij} \in \R^{(\nx+\enu) \times (\nx+\enu)} \\\hspace{3mm}\begin{bmatrix}W_{ij}-\lambda_{ij}a_{b_{ij}}^{2}J^{\top}e_i^{\top}e_iJ&*\\e_jJ&\lambda_{ij}\end{bmatrix}\succeq0
		\\\hspace{3mm}0\succeq\alpha I+\bar{A}+\bar{A}^{\top}+\sum_{i,j}W_{ij}+\zeta J^\top J\end{matrix*}\right.$.
		\\ \textbf{Solve } $\left\{\begin{matrix*}[l] P_k \in \mathbb{S}^{\nx+\enu}_{\succ 0} \\ P_k\bar{A}+\bar{A}^{\top}P_k+{\zeta_k^{-1}}PJ^\top JP\preceq -\alpha_kI\end{matrix*}\right.$.
		\\ \textbf{Set }$k=1$.
		\While {$\left|\alpha_k\zeta_k^{-1}-\alpha_{k-1}\zeta_{k-1}^{-1}\right|>\eps$}
		\\ \textbf{Solve} $\left\{\begin{matrix*}[l] (P_k,F_k,\mu_{{ij}_k},\alpha_k,\zeta_k)=\arg\!\max \text{\space}(\ln{\alpha}-\ln{\zeta})
		\\\ensuremath{\,\textrm{s.t.}}
		\hspace{3mm}{\kappa_{ij},\mu_{ij} \in \R_{>0}, \text{\space}P,U \in \mathbb{S}^{\nx+\enu}_{\succ 0},}
		\\ \hspace{7mm}{\Cc \in \R^{\enu \times \enu},\text{\space} \Bc,\Dc \in \R^{\enu \times \ny}}
		\\\hspace{7mm} {M_k\coloneqq [P\bar{A}+P_k\bar{B}(F-F_k)\bar{C}+P\bar{B}F_k\bar{C}]^\dagger+\alpha I}
		\\\hspace{7mm}G_{3_k}\coloneqq \BlkDiag{(-2\mu_{ij_k}+\mu_{ij})}_{i,j}
		\\\hspace{7mm}G_k\coloneqq \BlkDiag{G_1,G_{2},G_{3_k},-2U_k+U,-U,-\zeta I}
		\\\hspace{7mm}H_{2_k}\coloneqq \BlkDiag{\mu_{ij_k}}H_2
		\\\hspace{7mm}{X_k\coloneqq\begin{bmatrix}H_{2_k}&(P-P_k)U_k^{\top}& \bar{B}(F-F_k)\bar{C} &PJ^\top \end{bmatrix}}
		\\\hspace{7mm}\begin{bmatrix}
		M_k+\sum_{i,j}\kappa_{ij}J^{\top}e_j^{\top}e_jJ&\begin{matrix}*&*&*\end{matrix} \\
		\begin{matrix} H_{1}^{\top} \\H_{1}^{\top}\\ X_k^{\top} \end{matrix} & G_k
		\end{bmatrix} \preceq 0\end{matrix*}\right.$
		\\ \textbf{Set} $k+1 \leftarrow k$.
		\EndWhile
	\end{algorithmic}
\end{algorithm}}
{\subsection{Examples}}
\label{subsection:simulation}
{In this section,we illustrate the main results of Theorems~\ref{thm.Continuous} and~\ref{thm.aperiodic}.}

\begin{Ex}[Synthetic setting]
Consider system~\eqref{eq.sys-main} with the nominal matrices\footnote{These nominal matrices are chosen from \emph{Compl$_e$ib} library of MATLAB (\url{http://www.complib.de/}).}
\begin{align}
    \nonumber
    \begin{gathered}
        A=\begin{bmatrix}
        1.40&-0.21&6.71&-5.68\\
        -0.58&-4.29&0&0.67\\
        1.07&4.27&-6.65&5.89\\
        0.05&4.27&1.34&-2.10
        \end{bmatrix},\,
        B=\begin{bmatrix}
        0&0\\5.68&0\\1.14&-3.15\\1.14&0
        \end{bmatrix},\,
        C=\begin{bmatrix} 1&0\\0&1\\1&0\\-1&0\end{bmatrix}^\top.
    \end{gathered}
\end{align}
The uncertainty bounds are $\Mbdd=0.1(\mathds{1}_4^\top\otimes \mathds{1}_4)$ and $\Nbdd=0.1(\mathds{1}_2^\top\otimes \mathds{1}_4)$. Matrices $\Bc$, $\Cc$, and $\Dc$ are found from~\eqref{eq.continuous_NMI} by means of the aforementioned technique. In this example, we consider the desired output value as $\ydes = \begin{bmatrix} 9 & 10 \end{bmatrix}$. We first examine the result of Theorem~\ref{thm.Continuous}. For this purpose, we consider a nonlinear term in the form $k^\star(x)=\rhob/2\begin{bmatrix} sin(x_1(t))&\dots&sin(x_4(t))\end{bmatrix}$ in the dynamic~\eqref{eq.sys-main} and inspect the influence of amplitude~$\rhob$ on the desired regulation performance. Figure~\ref{fig.kb} compares the actual regulation error (i.e., deviation between the output and its desired value) in solid black line, and the predicted error by~\eqref{epsc} in dashed red line. 
				
{Next, we introduce a simulation setting to validate the theoretical bound~\eqref{eq.hmax_definition} in Theorem~\ref{thm.aperiodic}. While \eqref{eq.hmax_definition} anticipates that $\hb \leq 0.0286$ ensures the stability of the system under AETC, the numerical investigation shows that in this example the stability is guaranteed for higher values up to $\hb \leq 0.105$. It is, however, worth mentioning that the regulation error is not much influenced by $\hb$ as long as $\hb \leq 0.105$. This observation is also qualitatively aligned with the assertion of Theorem~\ref{thm.aperiodic} (cf.~\eqref{eq.eps_d.def} and its dependency on $\hb$ as defined in~\eqref{eq.f.def}).}

With regards to the triggering mechanism and its impact on the regulation error in Theorem~\ref{thm.aperiodic}, we vary the threshold level in the inequality~\eqref{eq.etc-structure} in the form {$\qt_0=\qt_1=\xi$}. The solid black line in Figure~\ref{fig.threshold} shows the impact of this variation of the pair $(\qt_0,\qt_1)$ through the variable~$\xi$ on the actual the regulation error. As anticipated by Theorem~\ref{thm.aperiodic}, the degradation of the regulation performance is dominated by the theoretical bound~\eqref{eq.eps_d.def} (red dashed line). Besides these error bounds, we also inspect the relation between the relative frequency of triggered events (in proportion to the total number of sampling instants) and the threshold level. This observation is depicted in blue dotted curve with the axis on the right-hand side of Figure~\ref{fig.threshold}. As expected, the increase of the threshold monotonically reduces the frequency of the triggering events.
\begin{figure}[!t]
	\centering
    \minipage{0.5\textwidth}
	\centering
    \includegraphics[width=1\textwidth]{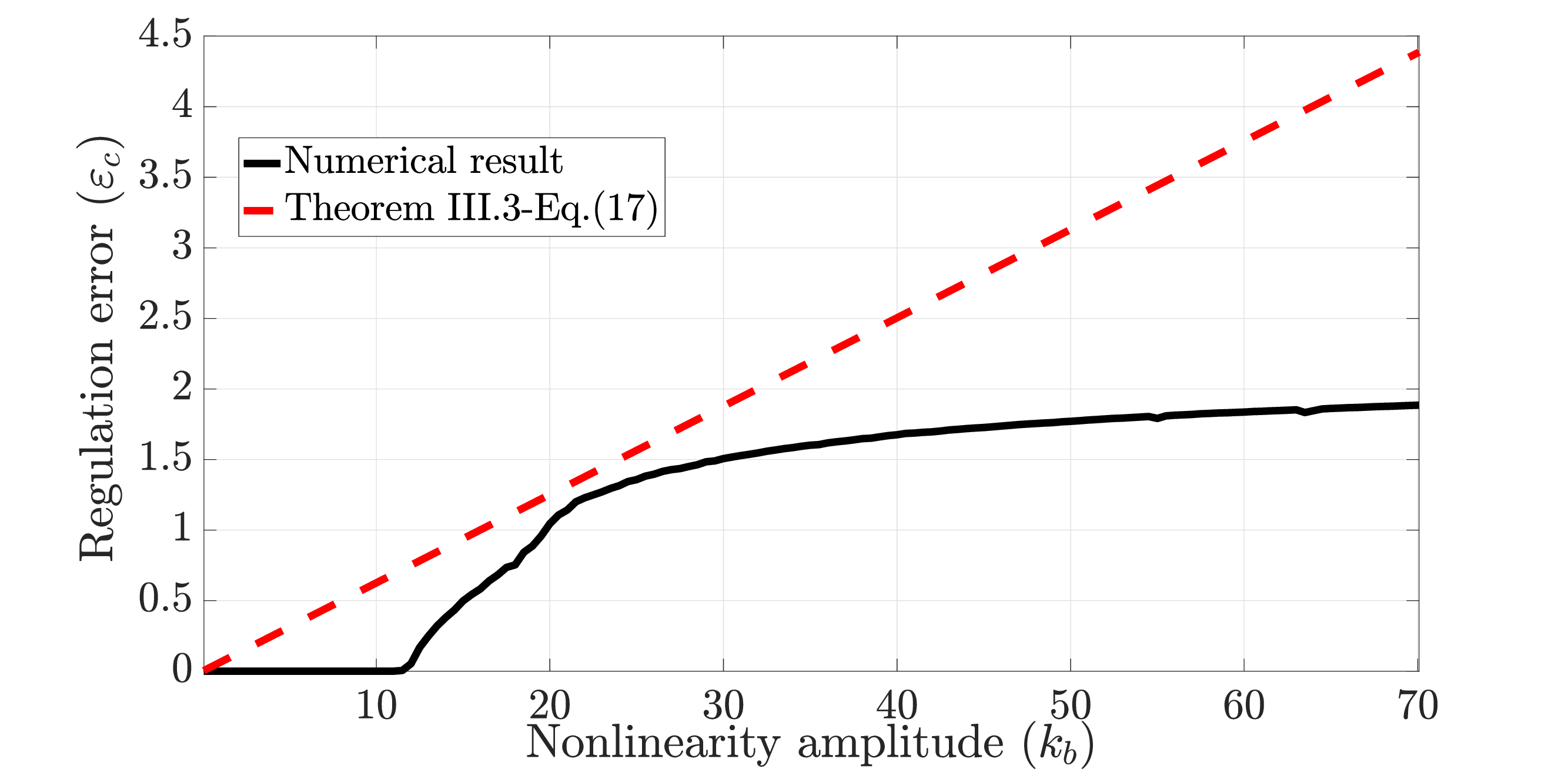}
    \caption{{\small Impact of the nonlinearity amplitude on theoretical bound~\eqref{epsc} in Theorem~\ref{thm.Continuous} and the actual numerical error.}}
	\label{fig.kb}
	\endminipage\hfill
    \minipage{0.5\textwidth}
    \centering
    \includegraphics[width=1\textwidth]{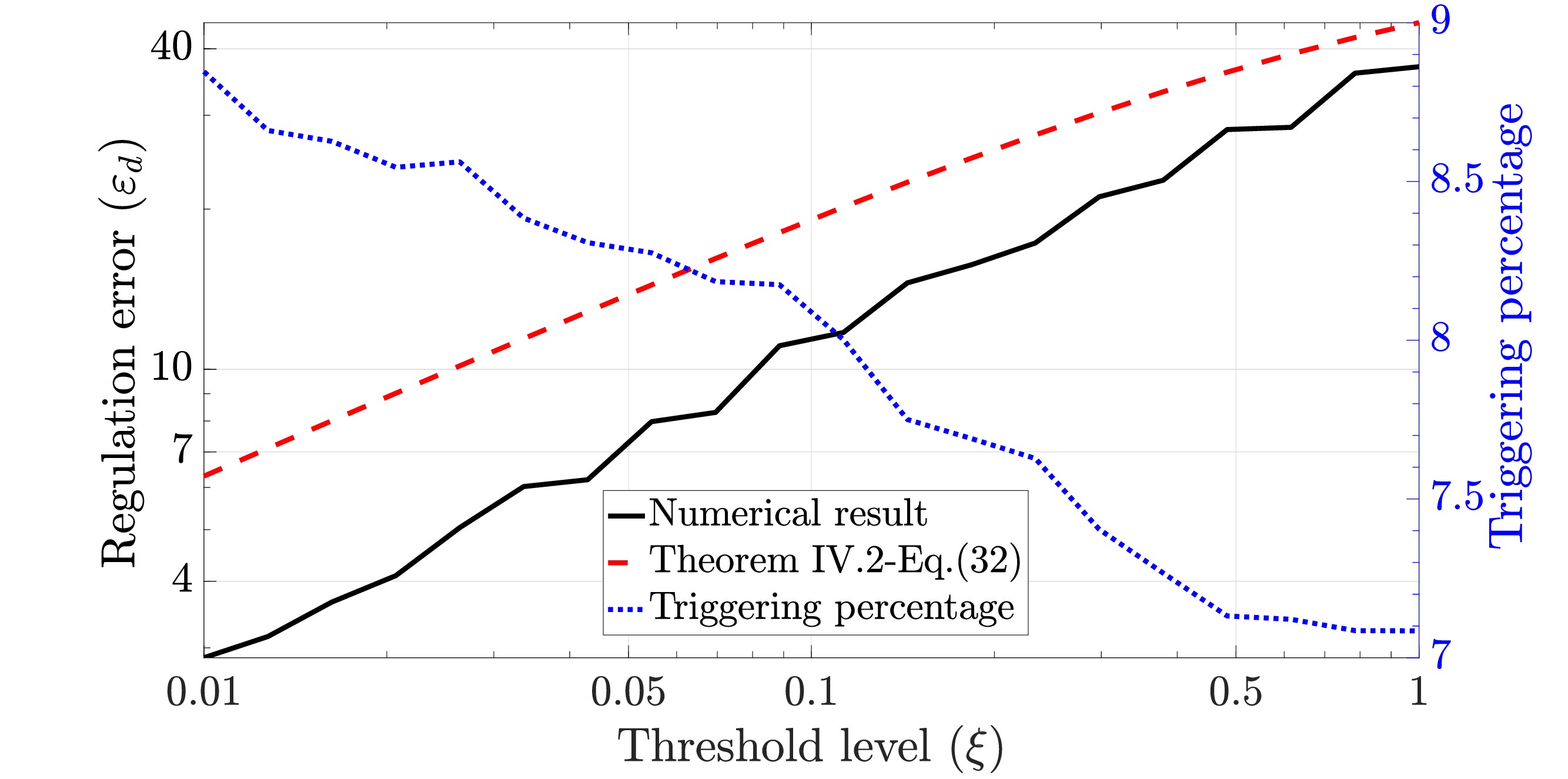}
    \caption{{\small Impact of threshold~\eqref{eq.etc-structure} on theoretical bound~\eqref{eq.eps_d.def} in Theorem~\ref{thm.aperiodic}, the actual numerical error, and the triggered events rate.}}
	\label{fig.threshold}
	\endminipage\hfill
\end{figure}
\end{Ex}
%==========================================================
\section{{Conclusion}}
\label{sec:conclusion}
%==========================================================
{In this article, we introduced an optimization-based framework to synthesize robust dynamic controllers in order to ensure the output regularization task for systems with uncertain and potentially nonlinear dynamics. To numerically solve such an optimization problem, a sequential parametric convex approximate algorithm was proposed. We further introduced a general sampling-based event-triggered technique that paves the way to implement the proposed controller in case of sampled measurements and discontinuous actuation updates. It is remarkable that the procedure of the triggering law is decoupled from control synthesis and the key parameters such as maximal inter-sampling time is explicitly computationally available. }

%===============================================================================
\bibliographystyle{plain}
\bibliography{ref1}
%===============================================================================

\end{document}